\documentclass{amsart}
\usepackage[utf8]{inputenc}
\usepackage{amssymb,bm,amsfonts, stmaryrd, amsmath, amsthm, enumerate, mathtools,comment}
\usepackage[all]{xy}
\makeatletter
\@namedef{subjclassname@2020}{\textup{2020} Mathematics Subject Classification}
\makeatother
\usepackage[colorlinks]{hyperref}
\hypersetup{
 colorlinks=true,
 linkcolor=blue,
 filecolor=magenta, 
 urlcolor=cyan,
}
\usepackage[shortlabels]{enumitem}
\usepackage{wrapfig,tikz, tikz-cd}
\usetikzlibrary{arrows}
\usepackage[capitalise]{cleveref}
\crefformat{equation}{(#2#1#3)}
\crefrangeformat{equation}{(#3#1#4--#5#2#6)}
\crefformat{enumi}{(#2#1#3)}
\crefrangeformat{enumi}{(#3#1#4--#5#2#6)}
\newtheorem{theorem}{Theorem}[section]
\newtheorem{lemma}[theorem]{Lemma}
\newtheorem{proposition}[theorem]{Proposition}
\newtheorem{corollary}[theorem]{Corollary}
\newcounter{intro}

\newtheorem{introthm}[intro]{Theorem}

\theoremstyle{definition}
\newtheorem{definition}[theorem]{Definition}

\newtheorem{example}[theorem]{Example}
\newtheorem{remark}[theorem]{Remark}

\newtheorem{chunk}[theorem]{}

\newtheorem*{ack}{Acknowledgements}

\theoremstyle{definition}

\newcommand{\Spec}{{\operatorname{Spec}}}

\newcommand{\Ext}{{\operatorname{Ext}}}
\newcommand{\Tor}{{\operatorname{Tor}}}

\renewcommand{\H}{\operatorname{H}}

\DeclareMathOperator{\pdim}{projdim}

\newcommand{\x}{\bm{x}}

\newcommand{\Supp}{{\operatorname{Supp}}}

\newcommand{\ann}{{\operatorname{ann}}}

\newcommand{\m}{\mathfrak{m}}
\newcommand{\p}{\mathfrak{p}}

\title[On liftings of modules of finite projective dimension]{On liftings of modules of finite projective dimension}

\author[N.~KC]{Nawaj KC}
\address{Department of Mathematics,
University of Nebraska, Lincoln, NE 68588, U.S.A.}
\email{nkc3@huskers.unl.edu}
\author[A.~Soto Levins]{Andrew J. Soto Levins}
\address{Department of Mathematics and Statistics,
Texas Tech University, Lubbock, TX 79409 U.S.A.}
\email{ansotole@ttu.edu}


\begin{document}

\begin{abstract}
    We introduce and study a notion of Serre liftable modules; these are modules that are liftable to modules of the maximal possible dimension over a regular local ring. We establish new cases of Serre's positivity conjecture over ramified regular local rings by proving it for Serre liftable modules. Furthermore, we show that the length of a nonzero Serre liftable module is bounded below by the Hilbert-Samuel multiplicity of the local ring. This establishes special cases of the Length Conjecture of Iyengar-Ma-Walker. 
\end{abstract}

\maketitle

\section*{Introduction.}\label{s_intro}

Suppose $R$ is a noetherian local ring and $Q$ is a regular local ring that surjects onto $R$. We are interested in lifting modules along this surjection. The classical notion of such liftings is as follows: a $Q$-module $L$ is a \textit{lift} of an $R$-module $M$ if $L \otimes_Q R \cong M$ as modules and $\Tor^Q_i(L, R) = 0$ for $i > 0$. In this case, the $R$-module $M$ is said to be a \textit{liftable} to $Q$; see \cite{ Buchsbaum/Eisenbud:1972, Dao:2007, Hochster:1975, Jorgensen:1999, Jorgensen:2003, Peskine/Szpiro:1973}. In this paper, we introduce and study a weaker notion of lifting $R$-modules to $Q$. 

Consider a finitely generated $Q$-module $L$ and let $M$ be the corresponding $R$-module $L \otimes_Q R$. By the dimension inequality of Serre \cite[Chapter V, Theorem 3]{Serre}, we have \[ \dim{Q}-\dim{L} \geq \dim{R}-\dim{M},\]and we say $L$ is a \textit{Serre lift} of $M$ if the inequality above is an equality. If $M$ admits such a lift, then $M$ is said to be \textit{Serre liftable} to $Q$. 

We prove that liftable modules are Serre liftable; see \cref{p_LiftableDimlyLiftable}. But there do exist unliftable modules that admit Serre lifts. Our main input in this paper is to deduce results about lengths and multiplicities of Serre liftable modules by using the intersection multiplicity pairing introduced by Serre; see \cref{intersectionmultiplicity}.  

Suppose $R$ is a regular local ring. Assume $M$ and $N$ are finitely generated $R$-modules such that $\ell_R(M\otimes_R N) < \infty$ and $\dim M + \dim N = \dim R$. In this case, Serre conjectured that the intersection multiplicity, $\chi^R(M, N)$, is strictly positive \cite[Chapter V, 4]{Serre}. Assume $R = Q/(f)$ where $Q$ is an unramified regular local ring and $f \in Q$. It is known that if $M$ is liftable to $Q$, then $\chi^R(M, N) > 0$ \cite{Buchsbaum/Eisenbud:1972}. One of our main results is to show that it suffices to assume $M$ is Serre liftable. 

We also establish uniform lower bounds on lengths of Serre liftable modules. Iyengar, Ma, and Walker have conjectured that if $M$ is a nonzero module of finite length and finite projective dimension over a local ring $R$, then $\ell_R(M) \geq e(R)$, where $e(R)$ is the Hilbert-Samuel multiplicity of $R$  \cite[Conjecture 1]{Iyengar/Ma/Walker:2022}. By Ma's thesis work \cite[Chapter V]{Ma:2014}, this conjecture is a generalization of the long-standing Lech's conjecture \cite{Lech:1960} for Cohen-Macaulay rings.  We prove that if $M$ is Serre liftable to an unramified regular local ring $Q$, then $\ell_R(M) \geq e(R)$. This result is new even for liftable modules. These two results are subsumed in the following theorem. 

\begin{introthm}
    \label{maintheorem} Let $R$ be a local ring and $M\neq 0$ an $R$-module that has a Serre lift to an unramifed regular local ring $(Q, \m)$. For a finitely generated $R$-module $N$ with $\ell(M\otimes_{R}N)<\infty$ and $\dim{M}+\dim{N}=\dim{R}$, one has
\[\ell(M\otimes_{R}N)\geq e(N).\]
When moreover $R = Q/(f)$ for $f \in \m\backslash\m^2$, $\chi^{R}(M,N)\geq e(N)$ also holds.
\end{introthm}

In the final section, we give an example of an unliftable module of finite length and finite projective dimension over a non-regular ring that admits a Serre lift; see \cref{mainExample}. We also give an example of a module of finite projective dimension over a non-regular ring that is not Serre liftable; see \cref{unliftableModule}. We do not know of any example of a module over a ramified regular ring that does not admit a Serre lift to an unramified regular ring; see \cref{mainquestion}.

\begin{ack}
    It is a pleasure to thank Mark Walker and Jack Jeffries for providing innumerable ideas and insights and their comments on preliminary drafts. We are thankful to Linquan Ma for several discussions, helping us clarify many of our results, and encouraging us to consider some key generalizations of our results. We thank  him and Mark Walker for pointing out the argument in \cref{unliftableModule}. We thank Josh Pollitz for his feedback on preliminary drafts, many fruitful discussions, and also the ideas behind \cref{obstructionforCI}, and Ryan Watson for many enjoyable discussions. We thank the referee for their extensive feedback.

    The first author was supported through NSF grants DMS-2200732 and DMS-2044833. This material is also based upon work supported by the National Science Foundation under Grant No. DMS-1928930 and by the Alfred P. Sloan Foundation under grant G-2021-16778, while the authors were in residence at the Simons Laufer Mathematical Sciences Institute (formerly MSRI) in Berkeley, California, during the Spring 2024 semester.
\end{ack}

\section{Serre liftable modules.}
\noindent Throughout this paper, all rings are noetherian local rings. We will say a local ring $(R, \m)$ is unramified if it is either of equal characteristic or it is of mixed characteristic $(0,p)$ for a prime integer $p$ with $p \in \m \setminus \m^2$. For basic results on Hilbert-Samuel multiplicities, we refer to \cite{Bruns/Herzog:1998, Serre}; alternatively, see \cite[Section 2]{Iyengar/Ma/Walker:2022} as we use similar notation. 

\begin{chunk} \label{intersectionmultiplicity}
    Let $R$ be a ring and let $M$ and $N$ be complexes such that ${\ell_R(\Tor_i^{R}(M,N)) < \infty}$ for all $i$. Assuming $\pdim_R(M)<\infty$, the \textit{intersection multiplicity} pairing is:\[ \chi^R(M, N) :=  \sum_{i \geq 0} (-1)^i \ell_R(\Tor_i^{R}(M,N)).\]When $R$ is an unramified regular local ring, and $M$ and $N$ are $R$-modules, Serre \cite[Chapter V, Theorem 1]{Serre} proved the following properties of this pairing: \label{serretheorem}
    \begin{enumerate}[(i)] 
 \item \textit{Non-negativity}: $\chi^{R}(M, N) \geq 0$, \item \textit{Vanishing:} If $\dim M + \dim N < \dim R$, then $\chi^{R}(M, N) = 0$, \item  \textit{Positivity:} If $\dim M + \dim N = \dim R$, then $\chi^{R}(M,N) > 0$.
\end{enumerate}
    Positivity remains an open conjecture for ramified regular local rings. Non-negativity (due to Gabber \cite{RobertsGabber:1998}) and the vanishing property (due to Gillet-Soul\'e \cite{Gillet/Soule:1987} and Roberts \cite{RobertsVanishing:1985} independently) have been proven for all regular rings. We utilize the following theorem of Serre (in equal characteristic) and Skalit (in the unramified mixed characteristic).

\begin{theorem} \cite[Theorem A]{Skalit:2015}, \cite[Ch.V, Thm 1 Complement]{Serre} \label{serreskalittheorem}
Suppose $R$ is an unramified regular local ring and $M$, $N$ are finitely generated $R$-modules such that $\ell_R(M \otimes_R N) < \infty$. If $ \chi^{R}(M,N) > 0 $, then $ \chi^{R}(M, N) \geq e(M) e(N). $   
\end{theorem}
\end{chunk}

\begin{chunk}
    Suppose $Q$ is a regular local ring. The dimension inequality of Serre \cite[Ch.V, Theorem 3]{Serre} establishes that for any two finitely generated modules $M$ and $N$ over $Q$, we have \[ \dim M + \dim N \leq \dim Q + \dim(M\otimes_Q N).\]
    \begin{definition}
        Suppose $Q$ surjects onto a local ring $R$, let $L$ be a finitely generated $Q$-module, and set $M = L\otimes_Q R$. We say $L$ is a \textit{Serre lift} of the $R$-module $M$ if \[ \dim{Q}-\dim{L} = \dim{R}-\dim{M}. \]If an $R$-module $M$ admits a Serre lift $L$, then $M$ is said to be \textit{Serre liftable} to $Q$.  
    \end{definition}

   If instead $L$ is a $Q$-module such that $L \otimes^{\mathsf L}_Q R \simeq M$, we say $L$ is a \textit{lift} of $M$. That is, $L \otimes_Q R \cong M$ as modules and $\Tor^Q_i(L, R) = 0$ for $i > 0$. We now establish liftable modules are Serre liftable. The following lemma is a consequence of the fact that $[N]=\Sigma(-1)^{i}[H_{i}(N)]$ in the Grothendieck group of complexes \cite[Corollary 6.6.1]{MR3076731}.

\begin{lemma} \label{l_key}
    Suppose $R$ is a local ring, $M$ is an $R$-module of finite projective dimension and $N$ is an $R$-complex with bounded homology. If $\ell(\Tor_{i}^{R}(M,N)) < \infty$ for all $i$, \[ \chi^{R}(M,N) = \sum_{i} (-1)^i\chi^{R}(M,\H_i(N)). \]
\end{lemma}

\begin{proposition} \label{p_LiftableDimlyLiftable}
    If a finitely generated $Q$-module $L$ is a lift of a nonzero $R$-module $M$, then $L$ is a Serre lift of $M$. 
\end{proposition}

\begin{proof}

    By the dimension inequality of Serre, we have \[ \dim{Q}-\dim{L} \geq \dim{R}-\dim{M}. \]Suppose $\dim M = r$, and choose $\x = x_1, \ldots, x_r$ in the maximal ideal of $R$ which is a parameter sequence on both $M$ and $R$.  Let $K_R = K(\x)$ be the Koszul complex on $\x$ over $R$. As $L$ is a lift of $M$, we have the following quasi-isomorphism of complexes \[ L \otimes^{\mathsf L}_Q K_R \simeq M \otimes^{\mathsf L}_R K_R.\]  As $\H_i(K_R)$ is annihilated by $(\x)$ as an $R$-module, $\Supp(\H_i(K_R)) \subseteq \Supp(R/(\x))$, whence $\dim \H_i(K_R) \leq \dim R/(\x)$ and $L \otimes_Q \H_i(K_R)$ has finite length for all $i$. By \cref{l_key}, we have \[ \chi^{Q}(L,K_R) = \sum_{i \geq 0} (-1)^i \chi^{Q}(L,\H_i(K_R)). \]For contradiction, assume \[ \dim{Q}-\dim{L} > \dim{R}-\dim{M}, \]i.e. $\dim L + \dim R/(\x) < \dim Q$, whence $\dim L + \dim \H_i(K_R) < \dim Q$ for all $i$. The vanishing theorem of Roberts and Gillet-Soul\'e (see \cref{intersectionmultiplicity}) implies $\chi^{Q}(L,K_R) = 0$. This is a contradiction as \[\chi^{Q}(L,K_R) = \chi^{R}( M,K_R) = e(\x;M) > 0\]by Serre's theorem \cite[Chapter IV, Theorem 1]{Serre}. This completes our proof.
\end{proof}

\begin{example}
    Hochster  \cite[Example 2]{Hochster:1975} gives an example of a module over a ramified regular local ring that does not lift to an unramified regular local ring $Q$. It is easy to show that this example admits a Serre lift. 
    
    Also observe that $R/(\x)$ for $\x$ any system of parameters on a local ring $R$ is an $R$-module that admits a Serre lift. If $R$ is not Cohen-Macaulay, such modules are unliftable as $\pdim_R (R/(\x)) = \infty$ by the New Intersection theorem; see \cite{Roberts:1987} or \cite[Proposition 6.2.4]{Roberts:1998}. 
\end{example}

\end{chunk}

\section{Main Results.}

\begin{theorem} \label{maintheorem} Let $R$ be a local ring and $M\neq 0$ an $R$-module that has a Serre lift to an unramifed regular local ring $(Q, \m)$. For a finitely generated $R$-module $N$ with $\ell(M\otimes_{R}N)<\infty$ and $\dim{M}+\dim{N}=\dim{R}$, one has
\[\ell(M\otimes_{R}N)\geq e(N).\]
When moreover $R = Q/(f)$ for $f \in \m\backslash\m^2$, $\chi^{R}(M,N)\geq e(N)$ also holds.
\end{theorem}

\begin{proof}
Let $L$ be a Serre lift of $M$ to an unramified regular local ring $Q$. One then has that $\dim{L}+\dim{N}=\dim{Q}.$ Thus $\chi^{Q}(L,N)>0$ by the positivity theorem, and hence Theorem \ref{serreskalittheorem} gives the second inequality below. The first inequality follows by Lichtenbaum \cite[Theorem 2]{Lichtenbaum:1966}. 
\begin{align*}
\ell(M\otimes_{R}N) = \ell(L\otimes_{Q}N) &= \chi^{Q}(L,N) + \chi_{1}^{Q}(L,N) \\
&\geq \chi^{Q}(L,N) \\
&\geq e(L)e(N) \\
&\geq e(N).
\end{align*}
This justifies the first part of the result. Since $R = Q/(f)$ and since $f$ is a nonzerodivisor on $Q$, we have $\Tor_i^Q(L,R) = 0$ for $i > 1$. Setting $H = \Tor_1^Q(L, R)$, we have  \[\chi^{R}(M,N) = \chi^{Q}(L,N) + \chi^{Q}(H,N) \geq \chi^{Q}(L,N) \geq e(N).\]The first equality is by \cref{l_key}, the first inequality follows due to the non-negativity theorem of Gabber \cite{RobertsGabber:1998}, and the second was observed above.
\end{proof}

\begin{corollary} \label{perfectCorollary}
    Suppose $M$ is a nonzero Cohen-Macaulay module of finite projective dimension over a local ring $R$. If $M$ admits a Serre lift to an unramified regular local ring, then $e(M) \geq e(R)$. 
\end{corollary}

\begin{proof}
Since $R$ admits a nonzero Cohen-Macaulay module of finite projective dimension, $R$ is Cohen-Macaulay by \cite[Theorem 6.2.3, Proposition 6.2.4]{Roberts:1998}. Observe that we may also assume $R$ has an infinite residue field and choose sufficiently general elements $\x = x_1, \ldots, x_r \in \m \setminus \m^2$ such that $e(M) = \ell_R(M/(\x)M)$ and $e(R) = e(R/(\x)R)$; see \cite[Corollary 4.6.10]{Bruns/Herzog:1998}. Then ${L \otimes_Q R/(\x) \cong M/(\x)M}$ and $\dim L + \dim R/(\x) = \dim Q$. Now apply \cref{maintheorem}. 
\end{proof}

\begin{remark}\label{mainquestion}
     For Serre's positivity conjecture, we may assume $R$ is a ramified regular local ring such that $R = Q/(f)$ for some unramified regular local ring $(Q, \m)$ and $f \in \m \setminus \m^2$. Furthermore, we may assume that the intersecting modules are of the form $R/\p$ for $\p \in \Spec R$. We do not know whether there exists a $\p \in \Spec R$ such that $R/\p$ does not Serre lift to $Q$. 
\end{remark}

\begin{chunk}

Following \cite[2.6]{Iyengar/Ma/Walker:2022}, a \textit{short complex} $F$ over a local ring $R$ is a complex \[ F: 0 \to F_d \to \cdots \to F_0 \to 0\] such that $d = \dim R$, each $F_i$ is a finite free $R$-module, and $\ell_R(\H(F)) < \infty$.  Roberts constructed a short complex $F$ with $ \chi(F) = \sum_{i\geq 0}(-1)^i \ell_R(\H_i(F)) < 0$ \cite{Roberts:1989}. For our purposes, a more suitable invariant is the Dutta multiplicity, denoted by $\chi_{\infty}(F)$; see \cite[4.17]{Iyengar/Ma/Walker:2022} and \cite{Dutta:1983, Kurano:2001a, Kurano/Roberts:2000, Roberts:1998} for more details. It has been conjectured by Kurano \cite[Conjecture 3.1]{Kurano:1993} that $\chi_{\infty}(F) > 0$ for $F$ a short complex with nonzero homology over a local ring. While this remains open in mixed characteristic, Iyengar, Ma, and Walker have conjectured an even tighter lower bound \cite[Conjecture 2]{Iyengar/Ma/Walker:2022}: over a local ring $R$ any short complex $F$ with nonzero homology satisfies $ \chi_{\infty}(F) \geq e(R).$ An affirmative answer to this conjecture implies Lech's conjecture holds in full generality \cite[Proposition 8.3]{Iyengar/Ma/Walker:2022}.
\end{chunk}

\begin{theorem}
Suppose $R$ is a local ring of prime characteristic $p>0$ that is $F$-finite. Suppose $F$ is a short complex over $R$ with $\H_0(F) \not= 0$. If the $R$-module $\H_0(F)$ Serre lifts to a regular local ring $Q$ of characteristic $p>0$, then $\chi_{\infty}(F) \geq e(R)$. 
\end{theorem}

\begin{proof}
Let the $Q$-module $L$ denote a Serre lift of $\H_0(F)$ and let $G \xrightarrow{\simeq} L$ denote its minimal $Q$-free resolution.  For $e \geq 0$, \[\chi^{Q}(F^eG \otimes_Q R) = \ell(\H_0(F^eF)) - \chi_1^{Q}(F^eG \otimes_Q R).\]We have $\dim{Q}-\dim{L}= \dim R = d$, and as $Q$ is regular, $\chi_{\infty}(L, R) = \chi^{Q}(L, R)$ by \cite[1.3 Remark]{Dutta:1983}. Taking limits, \[ \chi^{Q}(L, R) = \chi_{\infty}(L, R) = \lim_{e\to \infty} \frac{\ell(\H_0(F^eF))}{p^{ed}} - \lim_{e\to \infty} \frac{\chi_1^{Q}(F^eG \otimes_Q R)}{p^{ed}}.\]By the proof of \cite[Theorem 7.3.5]{Roberts:1998}, as $F$ is a short complex, the first term on the right hand side is equal to $\chi_{\infty}(F)$. By Lichtenbaum \cite[Theorem 2]{Lichtenbaum:1966} we know $\chi_1(F^eG \otimes_Q R) \geq 0$ and therefore the second limit is at least zero. This shows that $\chi_{\infty}(F) \geq \chi(L, R) > 0$. The proof is then complete by \cref{serreskalittheorem}. 
\end{proof}

\section{Examples.}

It is well-known that modules do not lift in general; see \cite{Hochster:1975,  Dao:2007, Dutta/Hochster/McLaughlin:1985, Jorgensen:1999, Singh/Miller:2000, Peskine/Szpiro:1973, Piepmeyer/Roberts:2005, Roberts/Srinivas:2003} for examples. In this section, we show that there exists an unliftable module of finite length and finite projective dimension over a hypersurface singularity that admits a Serre lift, and a module of finite projective dimension that is not Serre liftable. 

\begin{chunk} \label{obstructionforCI}
    We begin by outlining an obstruction to liftability of modules over complete intersection rings via the notion of cohomological operators due to Gulliksen \cite{Gulliksen:1974} and Eisenbud \cite{Eisenbud:1980}. Suppose $R$ is a local complete intersection ring and assume $R$ is complete. Then $R \cong Q/(f_1, \ldots, f_c)$ where $Q$ is a regular ring and $f_1, \ldots, f_c$ a regular sequence on $Q$. There exists a ring of cohomological operators $S = k[\chi_1, \ldots, \chi_c]$ that acts on $\Ext^R(M, k)$ for any finitely generated $R$-module $M$. Following \cite{Eisenbud:1980}, one can observe that if $M$ is an $R$-module that lifts to $Q$, then $(\chi_1, \ldots, \chi_c) \Ext_R(M, k) = 0$. 
\end{chunk}

\begin{example} \label{mainExample} Let $(R, \m)$ be the two dimensional local hypersurface \[ k[[x, y, z]]/(z^2-xy) \] defined over some field $k$. Let $M$ be the module with the following presentation: \[ d_1 \colon R^4 \xrightarrow[]{ \left(\!\begin{array}{cccc}
x&0&z&y\\
y&x&y&0
\end{array}\!\right)} R^2 .\]Computing support via the $0$th Fitting ideal, one can show $\Supp(M) = \{ \m \}$. This module also has finite projective dimension. Indeed, \[ 0 \xrightarrow{} R^{2} \, \xrightarrow{\left(\!\begin{array}{cc}
-y-z&-y\\
-y&0\\
x+y+z&y\\
-z&x-z
\end{array}\!\right)} R^{4} \,\xrightarrow{\left(\!\begin{array}{cccc}
x&0&z&y\\
y&x&y&0
\end{array}\right)} R^{2} .\]is an acyclic complex: it is easy to check that this is indeed an $R$-complex and, using the Buchsbaum-Eisenbud acyclicity criterion \cite{Buchsbaum/Eisenbud:1973}, one can show that this is acyclic. Therefore $\pdim_R(M) = 2$. 

\begin{proposition}
 In the notation of \cref{mainExample}, $M$ is Serre liftable to ${S = k[[x, y, z]]}$ but it does not lift to any regular ring.
\end{proposition}
\begin{proof}
    We claim that $L$ with the following presentation is a Serre lift of $M$: \[ \widetilde d_1 \colon S^4 \xrightarrow{\left(\begin{array}{cccc}
x&0&z&y\\
y&x&y&0
\end{array}\right)} S^2.\]Clearly $L \otimes_S R \cong M$, and computing the support via the $0$th Fitting ideal, we get $\dim L = 1$. Hence $\dim{S}-\dim{L} = \dim{R}-\dim{M} = 2$. 

Now we show  $M$ does not lift to any regular ring. If $M$ is liftable to some regular ring $Q$ along a surjection $\pi : Q \twoheadrightarrow R$, then it also lifts to the completion of $Q$ at its maximal ideal, denoted $\widehat{Q}$. Furthermore, since we can factor the map ${\widehat Q \twoheadrightarrow R}$ to ${\widehat Q \twoheadrightarrow k[[x, y, z]] \twoheadrightarrow R}$, we may assume $M$ lifts to the regular local ring $S = k[[x, y, z]]$. Let $E = k[\chi]$ be the ring of Eisenbud operator that acts on $\Ext_R(M, k)$. The action of $\chi$ on $\Ext_R(M, k)$ is independent of the choice of lift of the minimal free resolution $F$ of $M$ to $S$ \cite[Proposition 1.2]{Eisenbud:1980}. Naively lifting the $R$-free resolution in \cref{mainExample} above to $S$, we can compute that \[ \widetilde d^2 = \left( \begin{array}{cc}
z^2-xy&0\\
0&0
\end{array}\!\right) = (z^2-xy) \left( \begin{array}{cc}
1&0\\
0&0
\end{array}\!\right) \]Hence, $\chi$ acts on $\Ext_R^*(M, k)$ via the matrix $\left(\begin{array}{cc}1&0\\0&0\end{array}\right)$, therefore does not annihilate it. Therefore, by \cref{obstructionforCI}, $M$ cannot lift to any regular ring.
\end{proof}

\end{example}

\begin{example} \label{unliftableModule}
    In this example, we show there exists a module of finite length and finite projective dimension that is not Serre liftable to any regular local ring. Suppose $R$ is a complete hypersurface ring of equal characteristic. Suppose an $R$-module $M$ Serre lifts to some regular ring $Q$. Since $R$ is complete and of equal characteristic, we may assume so is $Q$. Further, as $R$ is a hypersurface, we can assume $R = Q/(g_1, \ldots, g_c)$ for $g_1, \ldots, g_c$ a regular sequence on $Q$ and that $g_1, \ldots, g_{c-1} \in \m \setminus \m^2$ where $\m$ is the maximal ideal of $Q$. 
    
    Let $L$ be a $Q$-module that is a Serre lift of $M$. Choosing a sufficiently general sequence of elements in $ g_1', \ldots, g_{c-1}' \in (g_1, \ldots, g_{c-1}) \setminus \m^2$, we can consider the regular ring $Q' = Q/(g_1', \ldots, g_{c-1}')$. Then the $Q'$-module $L' = L \otimes_Q Q'$ has dimension $\dim L - (c-1)$ and $L' \otimes_{Q'} R = M$. Therefore, if $M$ Serre lifts to some regular ring $Q$, we may assume it lifts to a regular ring $Q$ such that $R= Q/(f)$. 
    
    Let $k$ be a  field, and let $R = k[[x, y, z, w]]/(xw-yz)$ and $\p = (x, y)R$. In \cite{Dutta/Hochster/McLaughlin:1985}, Dutta-Hochster-McLaughlin constructued an $R$-module $M$ such that $\ell_R(M) = 15$, $\pdim_R M < \infty$, and $\chi^R(M, R/\p) = -1$. Observe that $M$ cannot lift to a regular ring. If $M$ is not Serre liftable, we are done. Otherwise, assume a $Q$-module $L$ is a Serre lift of $M$ where $Q = k[[x, y, z, w]]$. Consider the exact sequence of $Q$-modules \[ 0 \to \ann_{L}(f) \to L \xrightarrow{f} L \to M \to 0 .\]Note $M_1 \coloneqq \ann_L(f)$ is a nonzero $R$-module as $M$ is not liftable. Using the standard change of ring long exact sequence of Tors, we have $\pdim_R(M_1) < \infty$. On the other hand, since $\dim L = 1$, $\chi^Q(f; L) > 0$ whence $\ell_R(M) > \ell_R(M_1)$. Furthermore, \[ 0 = \chi^Q(L, R/\p) = \chi^R(M, R/\p) - \chi^R(M_1, R/\p) \]where the first equality follows by the vanishing theorem over regular rings and the second by \cref{l_key}. Therefore, $\chi^R(M_1, R/\p) = - 1$ and $\ell_R(M_1) < \ell_R(M)$. In particular, $M_1$ is not liftable. If $M_1$ is not Serre liftable, we have found our example. Otherwise, we run this argument to deduce the existence of an unliftable nonzero module $M_2$ of finite projective dimension such that $\ell_R(M_2) < \ell_R(M_1)$ and $\chi^R(M_2, R/\p) = -1$. Clearly, this process must terminate at some point giving us the desired example. 
\end{example}

\bibliographystyle{amsplain}
\bibliography{references}

\providecommand{\bysame}{\leavevmode\hbox to3em{\hrulefill}\thinspace}
\providecommand{\MR}{\relax\ifhmode\unskip\space\fi MR }
\providecommand{\MRhref}[2]{%
  \href{http://www.ams.org/mathscinet-getitem?mr=#1}{#2}
}
\providecommand{\href}[2]{#2}
\begin{thebibliography}{10}

\bibitem{Bruns/Herzog:1998}
Winfried Bruns and J{\"u}rgen Herzog, \emph{{C}ohen-{M}acaulay rings}, revised
  ed., Cambridge Studies in Advanced Mathematics, vol.~39, Cambridge University
  Press, Cambridge, 1998. \MR{1251956}

\bibitem{Buchsbaum/Eisenbud:1972}
David~A. Buchsbaum and David Eisenbud, \emph{Lifting modules and a theorem on
  finite free resolutions}, Ring theory ({P}roc. {C}onf., {P}ark {C}ity,
  {U}tah, 1971), Academic Press, New York-London, 1972, pp.~63--74. \MR{340343}

\bibitem{Buchsbaum/Eisenbud:1973}
\bysame, \emph{What makes a complex exact?}, J. Algebra \textbf{25} (1973),
  259--268. \MR{314819}

\bibitem{Dao:2007}
Hailong Dao, \emph{On liftable and weakly liftable modules}, J. Algebra
  \textbf{318} (2007), no.~2, 723--736. \MR{2371969}

\bibitem{Dutta:1983}
Sankar~P. Dutta, \emph{Frobenius and multiplicities}, J. Algebra \textbf{85}
  (1983), no.~2, 424--448. \MR{725094}

\bibitem{Dutta/Hochster/McLaughlin:1985}
Sankar~P. Dutta, M.~Hochster, and J.~E. McLaughlin, \emph{Modules of finite
  projective dimension with negative intersection multiplicities}, Invent.
  Math. \textbf{79} (1985), no.~2, 253--291. \MR{778127}

\bibitem{Eisenbud:1980}
David Eisenbud, \emph{Homological algebra on a complete intersection, with an
  application to group representations}, Trans. Amer. Math. Soc. \textbf{260}
  (1980), no.~1, 35--64. \MR{570778}

\bibitem{Gillet/Soule:1987}
H.~Gillet and C.~Soul\'{e}, \emph{Intersection theory using {A}dams
  operations}, Invent. Math. \textbf{90} (1987), no.~2, 243--277. \MR{910201}

\bibitem{Gulliksen:1974}
Tor~H. Gulliksen, \emph{A change of ring theorem with applications to
  {P}oincar{\'e} series and intersection multiplicity}, Math. Scand.
  \textbf{34} (1974), 167--183. \MR{364232}

\bibitem{Hochster:1975}
Melvin Hochster, \emph{An obstruction to lifting cyclic modules}, Pacific J.
  Math. \textbf{61} (1975), no.~2, 457--463. \MR{412169}

\bibitem{Iyengar/Ma/Walker:2022}
Srikanth~B. Iyengar, Linquan Ma, and Mark~E. Walker, \emph{Multiplicities and
  {B}etti numbers in local algebra via lim {U}lrich points}, Algebra Number
  Theory \textbf{16} (2022), no.~5, 1213--1257. \MR{4471041}

\bibitem{Jorgensen:1999}
David~A. Jorgensen, \emph{Existence of unliftable modules}, Proc. Amer. Math.
  Soc. \textbf{127} (1999), no.~6, 1575--1582. \MR{1476141}

\bibitem{Jorgensen:2003}
\bysame, \emph{Some liftable cyclic modules}, Comm. Algebra \textbf{31} (2003),
  no.~1, 493--504. \MR{1969236}

\bibitem{Kurano:1993}
Kazuhiko Kurano, \emph{An approach to the characteristic free {D}utta
  multiplicity}, J. Math. Soc. Japan \textbf{45} (1993), no.~3, 369--390.
  \MR{1219875}

\bibitem{Kurano:2001a}
Kazuhiko Kurano, \emph{Test modules to calculate {D}utta multiplicities}, J.
  Algebra \textbf{236} (2001), no.~1, 216--235. \MR{1808352}

\bibitem{Kurano/Roberts:2000}
Kazuhiko Kurano and Paul~C. Roberts, \emph{Adams operations, localized {C}hern
  characters, and the positivity of {D}utta multiplicity in characteristic
  {$0$}}, Trans. Amer. Math. Soc. \textbf{352} (2000), no.~7, 3103--3116.
  \MR{1707198}

\bibitem{Lech:1960}
Christer Lech, \emph{Note on multiplicities of ideals}, Ark. Mat. \textbf{4}
  (1960), 63--86. \MR{140536}

\bibitem{Lichtenbaum:1966}
Stephen Lichtenbaum, \emph{On the vanishing of {${\rm Tor}$} in regular local
  rings}, Illinois J. Math. \textbf{10} (1966), 220--226. \MR{188249}

\bibitem{Ma:2014}
Linquan Ma, \emph{The {F}robenius endomorphism and multiplicities}, Ph.D.
  thesis, University of Michigan, 2014.

\bibitem{Singh/Miller:2000}
Claudia~M. Miller and Anurag~K. Singh, \emph{Intersection multiplicities over
  {G}orenstein rings}, Math. Ann. \textbf{317} (2000), no.~1, 155--171.
  \MR{1760672}

\bibitem{Peskine/Szpiro:1973}
C.~Peskine and L.~Szpiro, \emph{Dimension projective finie et cohomologie
  locale. {A}pplications \`a la d\'{e}monstration de conjectures de {M}.
  {A}uslander, {H}. {B}ass et {A}. {G}rothendieck}, Inst. Hautes \'{E}tudes
  Sci. Publ. Math. (1973), no.~42, 47--119. \MR{374130}

\bibitem{Piepmeyer/Roberts:2005}
Greg Piepmeyer and Paul Roberts, \emph{Constructing modules of finite
  projective dimension with prescribed intersection multiplicities}, J. Algebra
  \textbf{294} (2005), no.~2, 569--589. \MR{2183365}

\bibitem{RobertsVanishing:1985}
Paul Roberts, \emph{The vanishing of intersection multiplicities of perfect
  complexes}, Bull. Amer. Math. Soc. (N.S.) \textbf{13} (1985), no.~2,
  127--130. \MR{799793}

\bibitem{Roberts:1987}
\bysame, \emph{Le th\'eor\`eme d'intersection}, C. R. Acad. Sci. Paris S\'er. I
  Math. \textbf{304} (1987), no.~7, 177--180. \MR{880574}

\bibitem{Roberts:1989}
\bysame, \emph{Intersection theorems}, Commutative algebra ({B}erkeley, {CA},
  1987), Math. Sci. Res. Inst. Publ., vol.~15, Springer, New York, 1989,
  pp.~417--436. \MR{1015532}

\bibitem{Roberts:1998}
\bysame, \emph{Multiplicities and {C}hern classes in local algebra}, Cambridge
  Tracts in Mathematics, vol. 133, Cambridge University Press, Cambridge, 1998.
  \MR{1686450}

\bibitem{RobertsGabber:1998}
Paul~C. Roberts, \emph{Recent developments on {S}erre's multiplicity
  conjectures: {G}abber's proof of the nonnegativity conjecture}, Enseign.
  Math. (2) \textbf{44} (1998), no.~3-4, 305--324. \MR{1659224}

\bibitem{Roberts/Srinivas:2003}
Paul~C. Roberts and V.~Srinivas, \emph{Modules of finite length and finite
  projective dimension}, Invent. Math. \textbf{151} (2003), no.~1, 1--27.
  \MR{1943740}

\bibitem{Serre}
Jean-Pierre Serre, \emph{Alg\`ebre locale. {M}ultiplicit\'es}, Lecture Notes in
  Mathematics, vol.~11, Springer-Verlag, Berlin-New York, 1965, Cours au
  Coll\`ege de France, 1957--1958, r\'edig\'e{} par Pierre Gabriel, Seconde
  \'edition, 1965. \MR{201468}

\bibitem{Skalit:2015}
Chris Skalit, \emph{Intersection multiplicity of {S}erre in the unramified
  case}, J. Algebra \textbf{428} (2015), 91--121. \MR{3314287}

\bibitem{MR3076731}
Charles~A. Weibel, \emph{The {$K$}-book}, Graduate Studies in Mathematics, vol.
  145, American Mathematical Society, Providence, RI, 2013, An introduction to
  algebraic $K$-theory. \MR{3076731}

\end{thebibliography}

\end{document}